%% file: degenerate-arXiv.tex
\let\nofiles\relax

\documentclass[twoside]{article}

\usepackage{subfigure}
\usepackage{mathrsfs}
\usepackage[sumlimits]{amsmath}
\usepackage{amsthm,amssymb,amsfonts}
\usepackage{graphicx,color}
\usepackage{algorithm}
\usepackage{algpseudocode}
\usepackage{titlesec}
\usepackage{float}
\usepackage{caption}
\usepackage{extarrows}

\usepackage[
linktocpage=true,
colorlinks=true,
linkcolor=blue,
citecolor=blue,
urlcolor=blue
]{hyperref}
\usepackage[numbers,sort&compress]{natbib}

\allowdisplaybreaks

\input{pagefmt.tex}

\newtheorem{definition}{Definition}[section]

\newtheorem{lemma}[definition]{Lemma}
\newtheorem{theorem}[definition]{Theorem}

\newtheorem*{remark*}{Remark}

\newtheoremstyle{noparens}%
{}{}%
{\itshape}{}%
{\bfseries}{\bf .}%
{ }%
{\thmname{#1}\thmnumber{ #2}\mdseries\thmnote{#3}}

\theoremstyle{noparens}

\theoremstyle{definition}
\newtheorem{remark}[definition]{Remark}

\numberwithin{equation}{section}
\numberwithin{figure}{section}

\titleformat*{\section}{\large \bfseries}
\titleformat*{\subsection}{\bfseries}
\titleformat*{\subsubsection}{\normalsize \bfseries}

\begin{document}
%
%
%
%



\title{The Stability of a Coupled Degenerate Wave System Under Boundary Control}%

\author{SUN Ya-nan, ZHANG Qiong}


\vskip10pt


\vskip-10pt

\begin{abstract}
	In this paper, we investigate a system composed of two degenerate wave equations which are connected at one point. By introducing some inequalities on the weighted spaces and employing the frequency domain method, we prove that the system is polynomially stable, which depends on the degree of the degeneracy.
\end{abstract}


\mr{ 35Q74, 35Q93, 93D20 }
\section{Introduction}
In this paper, we investigate the stability of a system consisting of two coupled degenerate wave equations. Specifically, one string is defined on the interval $(0,1)$ and degenerates at the point $x = 0$, while the other string is defined on the interval $(-1, 0)$ and degenerates at the point $x=-1$. These two strings are connected at the common point $x = 0$, and boundary control is only applied at the free endpoint of the right string (i.e., $x=1$). The specific mathematical description is given as follows,
\begin{equation}\label{sys}
	\begin{cases}
		u_{tt} - (a(x)u_x)_x = 0,&(x,t)\in(0,1)\times(0,\infty),\\
		w_{tt} - (b(x)w_x)_x = 0, &(x,t)\in(-1,0)\times(0,\infty),\\
		u(0,t) = w(0,t),&t\in(0,\infty),\\
		\lim_{x\to 0^+}(au_x)(x,t) = (bw_x)(0,t),&t\in(0,\infty),\\
		\gamma u(1,t) + u_t(1,t) + u_x(1,t) = 0,& t\in(0,\infty),\\
		w(-1,t) = 0,&t\in(0,\infty), \; \mbox{if}\; \mu_b<1,\\
		\lim_{x\rightarrow -1^+} (bw_x)(x,t) = 0,&t\in(0,\infty),\; \mbox{if}\; \mu_b\geq 1,\\
		u(x,0) = u_0(x),\;\;u_t(x,0) = u_1(x),&x\in(0,1),\\
		w(x,0) = w_0(x),\;\; w_t(x,0) = w_1(x),&x\in(-1,0),
	\end{cases}
\end{equation}
where $\gamma>0$, $a(x) = x^{\mu_a}$ (with $0\leq\mu_a<1$) and $b(x)=(x+1)^{\mu_b}$ (with $0\leq\mu_b<2$). The energy of \eqref{sys} is defined as
\begin{equation*}
	\begin{split}
		E(t) = \frac{1}{2}
		\left(\|u_t\|_{L^2(0,1)}^2
		+ \|\sqrt{a}u_x\|_{L^2(0,1)}^2+ \|w_t\|_{L^2(-1,0)}^2  + \|\sqrt{b}w_x\|_{L^2(-1,0)}^2 +\gamma a(1)|u(1)|^2\right),
	\end{split}
\end{equation*}
and
$$\frac{d}{dt}E(t) = - a(1)|u_t(1)|^2.$$
Recently, the degenerate wave equations have received much attention.  A common feature of these systems is that they can be naturally linked to differential operators with variable diffusion coefficients. While these operators are not uniformly elliptic, even though they are in general uniformly elliptic in compact subsets of the space domain, provided that these subsets are at a positive distance from the so-called zone of degeneracy. This degeneracy zone may occur either on a part of the boundary or on a sub-manifold of the space domain.

Alabau-Boussouira, Cannarsa and Leugering \cite{Alabau} studied controllability and observability issues for degenerate wave equations of the form
\begin{equation}\label{example}
	u_{tt}-(\alpha(x)u_x)_x = 0,\;\;(x,t)\in(0,1)\times(0,\infty),
\end{equation}
where $\alpha$ is positive on $(0,1)$ but vanishes at zero. The degeneracy of \eqref{example} at $x=0$ is measured by the parameter $\mu$ defined by
\begin{equation*}
	\mu = \sup_{0< x \leq 1}\frac{x|\alpha'(x)|}{\alpha(x)}.
\end{equation*}
One says that \eqref{example} degenerates weakly if $0\leq\mu<1$, strongly if $1\leq \mu<2$. They established observability inequalities for weakly and strongly degenerate equations respectively, which fails when $\mu \geq 2$. Then, the exact controllability of the weakly and strongly degenerate equations was proved by Hilbert uniqueness method. They also obtained the exponential stability of the system with linear boundary control. Bai and Chai \cite{Chai} proved the exact controllability of  the same system in domains with moving boundary, on which control acts.
Different from \cite{Alabau}, Gueye \cite{Gueye} concentrated on the degenerate wave equation with control acting on the degenerate boundary, and obtained the boundary exact controllability for $\mu<1$. Kogut, Kupenko and Leugering \cite{Kogut} focused on the case where the degenerate point is internal. They discussed how the defect at this damaged point affects the transmission conditions at the site.
Kogut et.al \cite{Wang} discussed some issues related to a special class of weighted Sobolev spaces and derived some Poincare's type inequalities for the weighted function $\alpha$. Some results will be listed in Section 2, which are helpful in proving our main result.

All the above works focus on a single system. For the coupled system, Kogut et al. \cite{network,network1} investigated the boundary observability and exact null controllability for weakly and strongly degenerate linear wave equation on a star-shaped planar network.
Salhi, Moumni and Tilioua \cite{couple,couple1} considered a system composed of two degenerate wave equations, which are coupled via velocity terms, with only one equation controlled by a boundary feedback. They proved the controllability and the exponential stability of this system by the multiplier method.
Furthermore, similar to the approach in this paper, Akil, Fragnelli and Issa \cite{two-strings} studied a coupled system of one degenerate and one non-degenerate string connected at endpoints: the left part is a wave equation degenerate at the left endpoint, and the right part is a non-degenerate wave equation. It is shown that exponential stability of the whole coupled system can be
achieved by applying control only at the right endpoint. Liu and Zhang \cite{LZ} showed the stability of an elastic string system with local Kelvin-Voigt damping. They assumed that the damping coefficient has a singularity at the interface
of the damped and undamped regions and behaves like $x^{\mu}$ near the interface. The polynomial or exponential stability, which depends on the parameter $\mu\in (0, 1]$, is obtained.

Furthermore, the topic of degenerate parabolic equation has received a lot of attention of many authors.
For the null controllability of one-dimensional parabolic system, Cannarsa, Ferretti and Martinez \cite{para} proved that, with respect to the ``damaged" point, ``one-side" control can only guarantee the controllability of weakly degenerate systems. For strongly degenerate systems, the degeneracy is so strong that ``two-sides" control is necessary to ensure the system's controllability. Moreover, there are many studies on the stability of a wave equation coupled with a degenerate parabolic equation. In \cite{Han 1}, Han, Wang and Wang proved that the semigroup corresponding to such a system has a decay rate of $O(t^{-\frac{3 -\mu}{2(1-\mu)}})$ as $t\to \infty$, where $\mu\in(0, 1)$. However, it is not the optimal decay rate. Tebou \cite{Tebou} improved the decay rate and established a better decay rate over different intervals with respect to the degree of the degeneracy for $\mu\in(0, \frac{3}{4})$. Later, by analyzing the norm of the resolvent operator along the imaginary axis, Han, Song and Yu \cite{Han 2} enhanced the decay rate to $t^{-\frac{2-\mu}{1-\mu}}$, which is consistent with the optimal result obtained in \cite{zuazua} for the constant-diffusion coefficient equation ($\mu=0$).

The rest of the paper is organized as follows. In Section 2, we introduce the weighted Sobolev space associated with the degenerate function and some inequalities on such spaces. Section 3 is devoted to the well-posedness of the system \eqref{sys}. In Section 4, we employ the frequency domain method to prove the polynomial stability of the system.  We use $C$ to denote generic positive constants that may vary from line to line (unless otherwise stated). $u'$ denotes the derivatives with respect to space and $\langle \cdot,\cdot\rangle$ denotes the $L^2$-inner product.
\section{Assumptions and Preliminaries}
\subsection{Assumptions}
Let $x_0$ and $\ell$ be given real numbers satisfying $-\infty<x_0<\ell<+\infty$, and let $\alpha:[x_0,\ell]\rightarrow\mathbb{R}$ be a function satisfying the following assumptions:
\begin{subequations}
	\begin{align}
		&(I)&&\alpha(x)>0,\;\forall\; x \in(x_0,\ell],\; \alpha(x_0) = 0;\\
		&(II)&&\mu = \sup_{x_0<x\leq \ell}{(x- x_0 )|\alpha'(x)|\over \alpha(x)} = \lim_{x\rightarrow x_0^+} {(x - x_0)|\alpha'(x)|\over \alpha(x)}<2;\label{mu-r}\\
		&(III)&&\alpha(x)\in C[x_0,\ell] \cap C^1(x_0,\ell].
	\end{align}
\end{subequations}

\begin{lemma}
	Let $\alpha:[x_0,\ell]\rightarrow\mathbb{R}$ be a function satisfying (I)-(III). Then\\
	(i)
	\begin{equation}\label{lemma-r}
		\alpha(x)\geq \alpha(\ell)\left({x - x_0\over \ell- x_0}\right)^{\mu},\;\;\forall\; x\in[x_0,\ell];
	\end{equation}
	(ii) ${\alpha}^{-1}(\cdot)\in L^1(x_0,\ell)$ for $0\leq\mu<1$.
\end{lemma}

\begin{remark}
	It is easy to check that $a(x)$ and $b(x)$ satisfy the assumptions (I)-(III). From \eqref{lemma-r}, we have $a^{-1}\in L^1(-1,0)$.
\end{remark}
\subsection{Functional spaces}
We now introduce some weighted Sobolev spaces that are naturally associated with function $\alpha(x)$.
We denote by $V_{\alpha}^1(x_0,\ell)$ the space defined as follows
\begin{equation*}
	V_{\alpha}^1(x_0,\ell) = \left\{ u\in L^2(x_0,\ell)\;
	\left|\begin{array}{l}
		u\mbox{ is locally absolutely continuous in } (x_0,\ell],\\
		\sqrt{{\alpha}}u'\in L^2(x_0,\ell)
	\end{array}\right. \right\}.
\end{equation*}
It is easy to see that $V_{\alpha}^1(x_0,\ell)$ is a Hilbert space with the scalar product
\begin{equation*}
	\langle u ,v\rangle_
	{V_{\alpha}^1(x_0,\ell)}
	=\int_{x_0}^{\ell}
	u(x)v(x) + {\alpha}(x)u'(x)v'(x) dx,\;\;
	\forall\;u,v\in V_{\alpha}^1(x_0,\ell),
\end{equation*}
and associated norm
\begin{equation*}
	\|u\|_{V_{\alpha}^1(x_0,\ell)}
	=\left(\int_{x_0}^{\ell} |u(x)|^2 + {\alpha}(x)|u'(x)|^2 dx\right)^{1\over2},\;\;
	\forall\;u\in V_{\alpha}^1(x_0,\ell).
\end{equation*}
We also introduce the subspaces  $V_{\alpha,0}^1(x_0,\ell)$ and $V_{\alpha}^2(x_0,\ell)$ of $V_{\alpha}^1(x_0,\ell)$ defined as
\begin{equation*}
	V_{{\alpha},0}^1(x_0,\ell) = \left\{ u \in V_{\alpha}^1(x_0,\ell)\; |\; u(\ell) = 0\right\}.
\end{equation*}
and
\begin{equation*}
	V_{\alpha}^2(x_0,\ell) = \left\{ u \in V_{\alpha}^1(x_0,\ell)\;|\; {\alpha}u'\in H^1(x_0,\ell)\right\}.
\end{equation*}

We list some useful properties related to the above functional spaces without proof. The complete proof processes can be found in \cite{Alabau,Wang,Kogut}.
\begin{lemma}\label{lemma1}
	Assume (I)-(III) hold. Then, \\
	(1) For every $u\in V_{{\alpha}}^1(x_0,\ell)$,
	\begin{equation}\label{lemma3-2}
		\left\|u-u(\ell )\right\|_{L^2(x_0,\ell )}^2 \leq {(\ell - x_0)^2\over (2-\mu )\alpha(\ell )}
		\left\|\sqrt{\alpha}u'\right\|^2_{L^2(x_0 ,\ell )}.
	\end{equation}
	(2) For every $u\in V_{{\alpha},0}^1(x_0,\ell)$,
	\begin{equation}\label{lemma3-3}
		\left\|u\right\|_{L^2(x_0,\ell)}^2 \leq C_{\alpha} \left\|\sqrt{\alpha}u'\right\|_{L^2(x_0,\ell)}^2,
	\end{equation}
	where
	\begin{equation}
		C_{\alpha} = \max\left\{ {(\ell -x_0)^2\over (2-\mu )\alpha(\ell )},
		{4(\ell -x_0)^{\mu}\over {\alpha}(\ell )}\right\}.
	\end{equation}
	(3) For every $u\in V_{\alpha}^1(x_0,\ell)$,
	\begin{equation}\label{xu2}
		\lim_{x\rightarrow x_0^+} (x-x_0)u^2(x)
		=0.
	\end{equation}
(4) For every $u\in V_{\alpha}^2(x_0,\ell)$,
\begin{equation}\label{xau'2}
	\lim_{x\rightarrow x_0^+}(x - x_0){\alpha}(x)|u'(x)|^2
	=0,
\end{equation}
and if ${\alpha}$ degenerates strongly,
\begin{equation}\label{au'}
	\lim_{x\rightarrow x_0^+} {\alpha}(x)u'(x)
	= 0
\end{equation}
(5) For every $u\in V_{\alpha}^2(x_0,\ell)$ and $\varphi\in V_{\alpha}^1(x_0,\ell)$, if ${\alpha}$ degenerates strongly,
\begin{equation}\label{lemma-au'phi}
	\lim_{x\rightarrow x_0^+} {\alpha}(x)\varphi(x)u'(x)
	=0.
\end{equation}
If ${\alpha}$ degenerates weakly, \eqref{lemma-au'phi} also holds when $\varphi(x_0) = 0$.
\end{lemma}
\section{Well-posedness}
Let us denote by $V_a^1(0,1)$  the space $V_{\alpha}^1(x_0,\ell)$ with $\alpha =a(x) $, $x_0 = 0$ and $\ell =1$, and by $V_b^1(-1,0)$ the space $V_{\alpha}^1(x_0,\ell)$ with $\alpha=b(x)$, $x_0 = -1$ and $\ell =0$.

If $\mu_b\in[0,1)$, we denote by $W_{b}^1(-1,0)$ the subspace of $V_{b}^1(-1,0)$ defined as
\begin{equation*}
W_{b}^1(-1,0) =\left\{u\in V_{b}^1(-1,0)\;|\;u(-1) =0\right\}.
\end{equation*}
If $\mu_b\in[1,2)$, on the other hand, $W_{b}^1(-1,0)$ coincides with $V_{b}^1(-1,0)$ itself.

Moreover, we set
$$W_{b}^2(-1,0) = V_{b}^2(-1,0) \cap W_{b}^1(-1,0).$$
Notice that $W_{b}^2(-1,0) = V_{b}^2(-1,0)$ when $\mu_b\in[1,2)$.

Now, we consider the Hilbert space
$$\mathcal H_1 = V_a^1(0,1) \times L^2(0,1)\times W_{b}^1(-1,0) \times L^2(-1,0),\;\;\;\;
\mathcal H_2 = V_a^1(0,1) \times L^2(0,1)\times V_{b}^1(-1,0) \times L^2(-1,0),$$
endowed with the norm, $\forall\; U=(u,v,w,y)\in\mathcal H_i$, $i=1,2$,
\begin{equation*}
\begin{split}
	\|U\|_{\mathcal H}^2 = \left\|v\right\|_{L^2(0,1)}^2
	+ \left\|y\right\|_{L^2(-1,0)}^2
	+\left\|\sqrt{a}u'\right\|_{L^2(0,1)}^2
	+\left\|\sqrt{b}w'\right\|_{L^2(-1,0)}^2
	+ \gamma a(1)\left|u(1)\right|^2.
\end{split}
\end{equation*}
Define an operator $\mathcal A_i:\mathcal D(\mathcal A_i)\subset \mathcal H_i\rightarrow\mathcal H_i$, $i=1,2$, by
\begin{equation}\label{def}
\mathcal A_i U = (v,(au')',y,(bw')'),\;\;\forall\;U = (u,v,w,y),
\end{equation}
with domain
\begin{equation*}
\mathcal D(\mathcal A_1) = \left\{(u,v,w,y)\in \mathcal H_1\;\left|\;\begin{array}{l}
	u\in V_a^2(0,1),\;v\in V_a^1(0,1),
	w\in W_{b}^2(-1,0),\; y\in W_b^1(-1,0),\\
	u(0) = w(0),\;
	\gamma u(1) + u'(1) + v(1) = 0,
	\lim_{x\to  0^+} (au')(x) = (bw')(0)
\end{array}\right.\right\},
\end{equation*}
and
\begin{equation*}
\mathcal D(\mathcal A_2) = \left\{(u,v,w,y)\in \mathcal H_2\;\left|\;\begin{array}{l}
	u\in V_a^2(0,1),\;v\in V_a^1(0,1),
	w\in V_{b}^2(-1,0),\; y\in V_b^1(-1,0),\\
	u(0) = w(0),\;
	\gamma u(1) + u'(1) + v(1) = 0,
	\lim_{x\to  0^+} (au')(x) = (bw')(0)
\end{array}\right.\right\}.
\end{equation*}
Let $U(t) = (u(t),v(t),w(t),y(t))$ and $U_0= (u_0,u_1,w_0,w_1)$, then the system \eqref{sys} can be rewritten as
\begin{equation*}
{d\over dt}U(t) = \mathcal A_iU(t),\;\;U(0)=U_0\in\mathcal H_i,\;i=1,2.
\end{equation*}
\begin{lemma}
The operator $\mathcal A_i$ generates a $C_0$-semigroup of contractions $T_i(t)$ on the energy space $\mathcal H_i$, and $0\in\rho(\mathcal A_i)$, the resolvent of the operator $\mathcal A_i$, $i=1,2$.
\end{lemma}
\begin{proof}
For every $U=(u,v,w,y)\in\mathcal D(\mathcal A_i)$, $i=1,2$ it follows from \eqref{def} that
$$\Re \langle \mathcal A_iU,U \rangle = - a(1)\left|v(1)\right|^2 \leq 0,$$
which implies that $\mathcal A_i$ is dissipative.

Next, we shall show that $0\in\rho(\mathcal A_i)$, $i=1,2$. Let $U\in\mathcal D(\mathcal A_1)$ satisfy $\mathcal A_1 U = 0$.  Then,
\begin{equation}
	\begin{cases}
		v=0, &x\in[0,1],\\
		(au')' = 0,&x\in[0,1],\\
		y = 0,&x\in[-1,0],\\
		(bw')' =0,&x\in[-1,0].
	\end{cases}
\end{equation}
Then, it follows from the transmission condition that
\begin{equation}\label{0-du}
	u' = (bw')(0)x^{-\mu_a},\;\;\forall\;x\in(0,1),
\end{equation}
and
\begin{equation}\label{0-dw}
	w' = (bw')(0)(1+x)^{-\mu_b},\;\;\forall\;x\in(-1,0).
\end{equation}
Since $w\in W_b^1(-1,0)$, it holds that $w(-1) = 0$. Then, for any $x\in(-1,0)$,
$$ w(x)= {(bw')(0)\over1-\mu_b}(1+x)^{1-\mu_b}. $$
By the transmission condition, one has, for any $x\in(0,1)$,
\begin{equation}
	\begin{split}
		u(x) =(bw')(0)\left({x^{1-\mu_a}\over 1-\mu_a} + {1\over1-\mu_b}\right).
	\end{split}
\end{equation}
Substituting the above results into the boundary condition, one has
$$\gamma u(1)+u'(1)+v(1) = (bw')(0) \left(1 + {\gamma\over1-\mu_a} +{\gamma\over1-\mu_b}\right) = 0,$$
which implies that $(bw')(0) = 0$, i.e., $U\equiv0$.

Let $U\in\mathcal D(\mathcal A_2)$ satisfy $\mathcal A_2U=0$, we have that \eqref{0-du} and \eqref{0-dw} remain valid. Since $1\leq\mu_b<2$, one has
\begin{equation*}
	\lim_{x\rightarrow -1^+} (bw')(x) = 0,
\end{equation*}
combining with the transmission condition, which implies that
$$au'\equiv 0,\;\forall\; x\in(0,1),$$and
$$bw'\equiv 0,\;\forall\; x\in(-1,0).$$
Then, it follows from $a(x)>0$ $(x\neq0)$, $b(x)>0$ $(x\neq-1)$ and the transmission condition that
$$u(x)\equiv u(1),\;\;\forall\;x\in(0,1),$$
and $$ w(x) \equiv u(1),\;\;\forall\; x\in(-1,0).$$
Substituting the above results into the boundary condition yields
$\gamma u(1) = 0$,
which implies that $U\equiv 0$.

Then, we prove that $\mathcal A_i\;(i=1,2)$ is surjective. For any fixed $F= (f_1,f_2,g_1,g_2)\in \mathcal H_1\setminus\{0\}$. Let $\mathcal A_1U = F$, i.e.,
\begin{subequations}
	\begin{align}
		&v=f_1,&&x\in[0,1],\label{0-a}\\
		&(au')'=f_2,&&x\in[0,1],\label{0-b}\\
		&y=g_1,&&x\in[-1,0],\label{0-c}\\
		&(bw')'= g_2,&&x\in[-1,0].\label{0-d}
	\end{align}
\end{subequations}
From \eqref{0-b} and the transmission condition, one has
\begin{equation}\label{0-u'}
	u'(x)={(bw')(0)\over a(x)}+{1\over a(x)}\int_0^x f_2(r)dr,\;\;\forall\;x\in(0,1),
\end{equation}
and
\begin{equation}\label{0-u}
	u(x) = u(0) + \int_0^x {1\over a(s)}\left((bw')(0)+\int_0^s f_2(r)dr\right) ds,\;\;\forall\;x\in(0,1).
\end{equation}
Since $w\in W_b^2(-1,0)$, it holds that $w(-1) =0$. Then,
\begin{equation}\label{0-w'}
	w'(x) = {(bw')(0)\over b(x)} - {1\over b(x)}\int_x^0 g_2(r)dr,\;\;\forall\;x\in(-1,0),
\end{equation}
and
\begin{equation}\label{0-w}
	w(x) = \int_{-1}^x {1\over b(s)} \left( (bw')(0)  -  \int_s^0 g_2(r)dr\right) ds,\;\;\forall\;x\in(-1,0).
\end{equation}
By the transmission conditions, one has that
\begin{equation}\label{0-u00}
	u(0) = w(0) = {(bw')(0)\over 1-\mu_b} - \int_{-1}^0{1\over b(s)} \int_s^0 g_2(r)dr ds.
\end{equation}
Substituting \eqref{0-a}, \eqref{0-u'} and \eqref{0-u} into the boundary conditions $\gamma u(1) + u'(1) +v(1) =0$ , one has that
\begin{equation}\label{0-au'00}
	\begin{split}
		(bw')(0) = &\left(1 + {\gamma\over1-\mu_b}+{\gamma\over1-\mu_a}\right)^{-1}
		\left(\gamma\int_{-1}^0 {1\over b(s)}\int_s^0 g_2(r) dr ds - f_1(1)\right.\\
		&\left.\quad-\int_0^1 f_2(r)dr
		- \gamma\int_0^1{1\over a(s)}\int_0^sf_2(r)dr ds \right).
	\end{split}
\end{equation}
Thus, $u(x)$ and $w(x)$ can be uniquely determined by \eqref{0-u}, \eqref{0-w}, \eqref{0-u00} and \eqref{0-au'00}.

Let $\mathcal A_2 U= F$, then \eqref{0-a}-\eqref{0-d}, \eqref{0-u'} and \eqref{0-u} remain valid. It follows from $\lim_{x\rightarrow -1^+} bw'= 0$ that
\begin{equation}
	(bw')(x) = \int_{-1}^x g_2(r)dr,\;\;\forall\;x\in(-1,0),
\end{equation}
and
\begin{equation}\label{0-w1}
	w(x) = w(0) - \int_x^0{1\over b(s)}\int_{-1}^s g_2(r) drds,\;\;\forall\;x\in(-1,0).
\end{equation}
By the transmission condition, one has that
\begin{equation}\label{0-au'0}
	\lim_{x\to0^+}(au')(x) = (bw')(0) = \int_{-1}^0 g_2(r)dr.
\end{equation}
Substituting \eqref{0-au'0} into \eqref{0-u'}, \eqref{0-u}, respectively, and combining with the boundary condition, we have that
\begin{equation}\label{0-u0}
	\begin{split}
		u(0) = -{1\over\gamma} \int_0^1 f_2(r)dr - \int_0^1 {1\over a(s)}\int_0^s f_2(r)dr ds
		- {1\over\gamma} f_1(1)-\left({1\over\gamma} + {1\over1-\mu_a}\right)\int_{-1}^0 g_2(r)dr.
	\end{split}
\end{equation}
In the end, $u(x)$ and $w(x)$ can be  uniquely determined by \eqref{0-u}, \eqref{0-w1}, \eqref{0-au'0} and \eqref{0-u0}.
\end{proof}
\begin{lemma}
$i\mathbb R\subset \rho(\mathcal A_i)$, $i=1,2$.
\end{lemma}
\begin{proof}
We first show that $i\mathbb R\subset\rho(\mathcal A_1)$. Suppose the statement of this lemma is not true. Then, there exist $\eta>0$ satisfying that $\{i\lambda\;|\;|\lambda|<\eta\}\subset\rho(\mathcal A_1)$. In this case, we can find a sequence $\{\lambda_n,U_n\}$ satisfying $|\lambda_n|\to\eta$, $|\lambda_n|<\eta$ and $\|U_n\|_{\mathcal H_1} = 1$, such that $$(i\lambda_n - \mathcal A_1)U_n = o(1),$$
which is equivalent to
\begin{subequations}
	\begin{align}
		&i\lambda_n u_n - v_n = o(1),&&\mbox{in}\;V_a^1(0,1),\label{lemma-a}\\
		&i\lambda_n v_n - (au_n')' = o(1),&&\mbox{in}\;L^2(0,1),\label{lemma-b}\\
		&i\lambda_n w_n - y_n = o(1),&&\mbox{in}\;W_b^1(-1,0),\label{lemma-c}\\
		&i\lambda_ny_n - (bw_n')' = o(1),&&\mbox{in}\;L^2(-1,0).\label{lemma-d}
	\end{align}
\end{subequations}
By the dissipativeness of the operator $\mathcal A_1$,
\begin{equation}\label{lemma-diss}
	\Re\langle \mathcal A_1 U_n , \; U_n\rangle = - a(1)|v_n(1)|^2 = o(1).
\end{equation}
From \eqref{lemma-a} and the boundary condition, one has that
\begin{equation}\label{lemma-bound}
	|u_n(1)| = \lambda_n^{-1}o(1),\;\; |u'_n(1)| \leq  o(1).
\end{equation}
Multiplying \eqref{lemma-b} by $2xu_n'$ and integrating over $(0, 1)$ and using lemma \ref{lemma1}, we obtain
\begin{equation}
	\begin{split}
		\|v_n\|_{L^2(0,1)}^2 + (1-\mu_a)\|\sqrt a u_n'\|_{L^2(0,1)}^2-a(1)|u_n'(1)|^2 -|v_n(1)|^2= o(1).
	\end{split}
\end{equation}
Notice that $0\leq\mu_a<1$, it follows from \eqref{lemma-diss} and \eqref{lemma-bound}  that
\begin{equation}\label{lemma-u}
	\|v_n\|_{L^2(0,1)} = o(1),\; \|\sqrt{a}u_n'\|_{L^2(0,1)} = o(1).
\end{equation}
Multiplying \eqref{lemma-d} by $2(1+x)w_n'$ and integrating over $(-1,0)$ and using lemma \ref{lemma1} again, we obtain
\begin{equation}\label{lemma-w}
	\begin{gathered}
		\|y_n\|_{L^2(-1,0)}^2
		+ (1-\mu_b)\|\sqrt{b}w_n'\|^2_{L^2(-1,0)}
		- b(0)|w_n'(0)|^2-|y_n(0)|^2
		= o(1).
	\end{gathered}
\end{equation}
It is clear that $a^{-1}\in L^1(0,1)$ since $0\leq \mu_a<1$.Then, we deduce from \eqref{lemma-a} that
\begin{equation}\label{lemma-v0}
	\begin{split}
		|v_n(0)| \leq &|\lambda_n|\left|\int_0^1 u_n' dx\right| + |v_n(1)|\\
		\leq& |\lambda_n|\left(\int_{0}^1 {1\over a} dx\right)^{1\over2}
		\left(\int_0^1 a|u_n'|^2dx\right)^{1\over2} + |v_n(1)|\\
		= & {\lambda_n\over\sqrt{1-\mu_a}}
		\|\sqrt{a}u_n'\|_{L^2(0,1)} + |v_n(1)|,
	\end{split}
\end{equation}
by the transmission condition, \eqref{lemma-diss}, \eqref{lemma-u} and $|\lambda_n|<\eta$,  which implies that
\begin{equation}\label{lemma-y0}
	|y_n(0)| = |v_n(0)| \leq o(1).
\end{equation}
Furthermore, since $b(0) =1$ and using the transmission condition, it holds that
\begin{equation}\label{lemma-bw'0}
	b(0)|w_n'(0)|^2 = |b(0)w_n'(0)|^2 = \left|\lim_{x\to 0^+}(au_n')(x)\right|^2.
\end{equation}
Then, it follows from $|\lambda_n|<\eta$, \eqref{lemma-b}, \eqref{lemma-bound} and \eqref{lemma-u} that
\begin{equation}\label{lemma-au'0}
	\begin{split}
		\left|\lim_{x\to 0^+}(au_n')(x)\right|
		\leq& \left|\int_0^1 (au')' dx\right| + |a(1)u_n'(1)|\\
		\leq& |\lambda_n|\left|\int_0^1 v_n dx\right| + |a(1)u_n'(1)|\\
		\leq& o(1).
	\end{split}
\end{equation}
From \eqref{lemma-w}--\eqref{lemma-au'0}, we have
\begin{equation}\label{lemma-w'}
	\|y_n\|_{L^2(-1,0)}^2 + (1-\mu_b)\|\sqrt{b}w_n'\|_{L^2(-1,0)}^2 = o(1).
\end{equation}
Due to $0\leq\mu_b<1$, one can directly deduce from \eqref{lemma-w'} that
$$\|y_n\|_{L^2(-1,0)}^2=o(1), \;\;
\|\sqrt{b}w_n'\|_{L^2(-1,0)}^2 = o(1),$$
which is a contradiction to $\|U_n\|_{\mathcal H_1} = 1$.

Next, we shall show that $i\mathbb R\subset\rho(\mathcal A_2)$ by the same method. In this case, \eqref{lemma-c} holds in $V_b^1(-1,0)$. Repeating the above arguments yields that \eqref{lemma-w'} remains true. Then, multiplying \eqref{lemma-d} by $w_n$ and integrating over $(-1,0)$, we obtain
\begin{equation}
	\|\sqrt{b}w_n'\|^2 - \|y_n\|^2 + b(0)w_n(0)w_n'(0)= o(1).
\end{equation}
By \eqref{lemma-a}, \eqref{lemma-v0}, \eqref{lemma-au'0} and the transmission conditions, it holds that
$$|b(0)w_n(0)w_n'(0)|\leq \left|\lim_{x\to0^+}a(x)u_n'(x)\right||u_n(0)|\leq \lambda_n^{-1}o(1).$$
Here $|u_n(0)|$ is bounded. Indeed,
\begin{equation}\label{lemma-u0}
	|u_n(0)|\leq |u_n(1)|+\int_0^1|u_n'| dx\leq |u_n(1)| +
	C\|\sqrt{a}u_n'\|_{L^2(0,1)} \leq o(1).
\end{equation}
Combining with \eqref{lemma-w'} and \eqref{lemma-u0}, it implies that
$$(2-\mu_b)\|\sqrt{b}w_n'\|_{L^2(-1,0)}^2 = o(1). $$
Then, $\|U_n\|_{\mathcal H_2} = o(1)$, which leads to a contradiction.
\end{proof}
\begin{lemma}\label{Tomilov}
\cite{Tomilov} Assume that $\mathcal A$ generates a $C_0$-semigroup $e^{\mathcal At}$ of contraction on a Hilbert space $\mathcal H$ and satisfies $i\mathbb R \subset\rho(\mathcal A)$. Then, $e^{\mathcal At}$ is polynomially stable with order $\theta$, if and only if
\begin{equation}
	\limsup_{|\omega|\rightarrow\infty}
	\omega^{-{1\over \theta}} \|(i\omega-\mathcal A)^{-1}\|_{\mathcal L(\mathcal H)} < \infty
\end{equation}
\end{lemma}
\section{The proof of main result}
In this section, we shall use the frequency domain method to prove that the system \eqref{sys} is polynomially stable, which depends only on $\mu_a$.

\begin{theorem}
There is a constant $C>0$ such that
\begin{equation}
	\|T_i(t)U_0\|_{\mathcal H}\leq t^{-{2-\mu_a\over 2}}\|U_0\|_{\mathcal D(\mathcal A_i)},\;\;\mbox{for each }U_0\in\mathcal D(\mathcal A_i),\;i=1,2.
\end{equation}
\end{theorem}
\begin{proof}
For clarity, we only provide the detailed argument  for the case $i=1$. The case $i=2$ follows analogously. By lemma \ref{Tomilov}, it is sufficient to show that there exists a constant $r>0$ such that
\begin{equation}\label{main}
	\inf_{\|U\|_{\mathcal H}=1,\;\lambda\in\mathbb R}\|\lambda_n^{\delta}(i\lambda_n - \mathcal A_1)U_n\| \geq r,\qquad\forall\;U\in\mathcal D(\mathcal A_1),
\end{equation}
where $\delta = {2\over 2-\mu_a}$. By contradiction, we suppose that \eqref{main} is not true. Then, there exists a sequence $\{\lambda_n, U_n\}\subset\mathbb R \times \mathcal D(\mathcal A_1)$ with $\|U_n\|_{\mathcal H_1} = 1$ and $\lambda_n\rightarrow \infty$, such that
\begin{equation}
	\|\lambda_n^{\delta}(i\lambda_n - \mathcal A_1)U_n\| = o(1),
\end{equation}
which is equivalent to, as $n\rightarrow\infty$,
\begin{subequations}
	\begin{align}
		&\lambda_n^{\delta}(i\lambda_n u_n - v_n) = o(1),\;\;&&\mbox{in}\;V_a^1(0,1),\label{a}\\
		&\lambda_n^{\delta}(i\lambda_n v_n - (au_n')')= o(1),\;\;&&\mbox{in}\;L^2(0,1),\label{b}\\
		&\lambda_n^{\delta}(i\lambda_n w_n - y_n) = o(1),\;\;&&\mbox{in}\;W_b^1(-1,0),\label{c}\\
		&\lambda_n^{\delta}(i\lambda_n y_n - (bw_n')') = o(1),\;\;&&\mbox{in}\;L^2(-1,0).\label{d}
	\end{align}
\end{subequations}
By dissipativeness of the operator $\mathcal A_1$, one has
$$\Re\langle (i\lambda_n - \mathcal A_1)U_n , U_n \rangle = -a(1)|v_n(1)|^2 = \lambda_n^{-\delta}o(1).$$
Combining with \eqref{a} and the boundary condition, it holds that
\begin{equation}\label{bound-o(1)}
	u_n(1) = \lambda_n^{-1-{\delta\over 2}}o(1),\;\;
	\left|u_n'(1)\right| \leq \gamma|u_n(1)| + |v_n(1)| = \lambda_n^{-{\delta\over2}}o(1).
\end{equation}
In what follows, we shall prove that
$$\|v_n\|_{L^2(0,1)},\;
\|\sqrt{a}u_n'\|_{L^2(0,1)},\;
\|y_n\|_{L^2(-1,0)},\;
\|\sqrt{b}w_n'\|_{L^2(-1,0)}
=o(1),$$
which contradicts $\|U_n\|_{\mathcal H_1} = 1$.

Taking $L^2(0,1)$ inner product of \eqref{b} with $2xu_n'$, then using lemma \ref{lemma1}, one has
\begin{equation*}
	\begin{gathered}
		\|v_n\|_{L^2(0,1)}^2 + (1-\mu_a)\int_{0}^1 a|u_n'|^2 dx - |v_n(1)|^2
		-a(1)|u_n'(1)|^2
		=\lambda_n^{-\delta}o(1).
	\end{gathered}
\end{equation*}
From $0\leq \mu_a<1$ and \eqref{bound-o(1)}, it holds that
\begin{equation}\label{r}
	\|v_n\|_{L^2(0,1)} = \lambda_n^{-{\delta\over 2}}o(1),\;\;
	\|\sqrt{a}u_n'\|_{L^2(0,1)} = \lambda_n^{-{\delta\over2}}o(1).
\end{equation}
Multiplying \eqref{b} and \eqref{d} by $u_n$ and $w_n$ separately, adding the results yields
\begin{equation}\label{add}
	\begin{gathered}
		\|\sqrt{a}u_n'\|_{L^2(0,1)}^2+\|\sqrt{b}w_n'\|_{L^2(-1,0)}^2
		+ \lim_{x\to0^+} a(x)u_n'(x)u_n(x)+ \lim_{x\to-1^+}b(x)w_n'(x)w_n(x)\\
		- \|v_n\|_{L^2(0,1)}^2-\|y_n\|_{L^2(-1,0)}^2
		- a(1)u_n'(1)u_n(1)- b(0)w_n'(0)w_n(0)= \lambda_n^{-\delta} o(1).
	\end{gathered}
\end{equation}
Then, it follows from transmission condition, the boundary conditions and \eqref{bound-o(1)} that
\begin{equation}\label{0}
	\begin{gathered}
		\lim_{x\to0^+} a(x)u_n'(x)u_n(x)+\lim_{x\to-1^+}b(x)w_n'(x)w_n(x)
		- a(1)u_n'(1)u_n(1)- b(0)w_n'(0)w_n(0)
		=\lambda_n^{-\delta}o(1).
	\end{gathered}
\end{equation}
Thus, substituting \eqref{r} and \eqref{0} into \eqref{add}, one has
\begin{equation}\label{left}
	\|\sqrt{b}w_n'\|_{L^2(-1,0)}^2
	-\|y_n\|_{L^2(-1,0)}^2
	=\lambda_n^{-\delta}o(1).
\end{equation}
Taking $L^2(-1,0)$ inner product of \eqref{c} with $2(x+1)w_n'$, and using lemma \ref{lemma1} and \eqref{left}, we obtain
\begin{equation}\label{right}
	(2-\mu_b)\|y_n\|^2_{L^2(-1,0)}
	-|y_n(0)|^2 - b(0)|w_n'(0)|^2 = \lambda_n^{-\delta}o(1).
\end{equation}

We claim that $y_n(0) = o(1)$ and $b(0)|w_n'(0)|^2= o(1)$. Indeed, it holds that
\begin{equation}\label{trans0}
	y_n(0) = v_n(0),\;\;\;b(0) |w_n'(0)|^2  = \left|\lim_{x\to 0^+}a(0)u_n'(0)\right|^2.
\end{equation}
Thus, it is sufficient to prove that
\begin{equation*}
	\lim_{x\to 0^+}a(x)u_n'(x) = o(1),\;\; v_n(0) = o(1).
\end{equation*}

Let $\xi_n = \lambda_n^{-\beta}$, where $\beta = {2\over2-\mu_a}>0$ is a positive number. Then,  there exists a sequence $\{\zeta_n\}$ satisfying $\xi_n\leq \zeta_n \leq 2\xi_n$ such that
\begin{equation}
	|(au')(\zeta_n)| =  \min_{\xi_n\leq x\leq 2\xi_n}|(au_n')(x)|\leq \xi_n^{-1}\int_{ \xi_n}^{2\xi_n} (au_n')(x)dx.
\end{equation}
where, by Holder's inequality,
\begin{equation}
	\begin{split}
		\int_{\xi_n}^{2\xi_n} (au_n')(x)dx
		\leq\left(\int_{\xi_n}^{2\xi_n} a  dx\right)^{1\over2}
		\left(\int_0^1 a|u_n'|^2dx\right)^{1\over2}
		\leq C\xi_n^{1+\mu_a\over2}\|\sqrt{a}u_n'\|_{L^2(0,1)}.
	\end{split}
\end{equation}
Due to \eqref{r}, it is easy to obtain that
\begin{equation}
	|(au_n')(\zeta_n)|\leq \lambda_n^{{\beta(1-\mu_a)\over2}-{\delta\over2}}o(1).
\end{equation}
Thus, it follows from \eqref{b} that
\begin{equation}\label{trans1}
	\begin{split}
		\left|\lim_{x\to 0^+}a(x)u_n'(x)\right|
		\leq &|(au_n')(\zeta_n)|+|\lambda_n|\left|\int_{0}^{\zeta_n} v_n dx\right|\\
		\leq & |(au_n')(\zeta_n)|+ \zeta_n^{1\over2}\lambda_n\|v_n\|_{L^2(0,1)}\\
		\leq & \lambda_n^{{\beta(1-\mu_a)\over2}-{\delta\over2}}o(1)
		+\lambda_n^{1-{\beta\over2}-{\delta\over2}}o(1).
	\end{split}
\end{equation}
On the other hand, there exists a sequence $\{\eta_n\}$ satisfying $\xi_n\leq \eta_n\leq 2\xi_n$ such that
\begin{equation}
	\begin{split}
		|v_n(\eta_n)| = \min_{\xi_n\leq x\leq 2\xi_n}|v_n(x)|
		\leq \xi_n^{-{1\over2}}\|v_n\|_{L^2(0,1)}
		=\lambda_n^{{\beta\over2}-{\delta\over2}}o(1).
	\end{split}
\end{equation}
Notice that $a^{-1}\in L^1(0,1)$ since $0\leq \mu_a <1$, one has
\begin{equation}\label{trans2}
	\begin{split}
		|v_n(0)|
		&\leq \left|\int_{0}^{\eta_n} v_n'dx\right| + |v_n(\eta_n)|\\
		&\leq |\lambda_n|\left(\int_{0}^{\eta_n}{1\over a(x)}dx\right)^{1\over2}
		\left(\int_{0}^{\eta_n}
		a(x)|u_n'|^2dx\right)^{1\over2}
		+|v_n(\eta_n)|\\
		&\leq \lambda_n^{1-{\beta(1-\mu_a)\over2}-{\delta\over2}}o(1)
		+\lambda_n^{{\beta\over2}-{\delta\over2}}o(1).
	\end{split}
\end{equation}
Substituting \eqref{trans1} and \eqref{trans2} into \eqref{right}, it follows from $\beta= \delta = {2\over2-\mu_a}$ that
\begin{equation}\label{left-0}
	(2-\mu_b)\|y_n\|^2_{L^2(-1,0)}\leq o(1).
\end{equation}
Then, it is easy to check that $\|U_n\|_{\mathcal H_1} =o(1)$ by \eqref{r} and \eqref{left-0}.

For the case $i=2$, \eqref{c} holds in the space$V_b^1(-1,0)$. Repeating the above arguments, we can conclude the same result. The proof is completed.
\end{proof}
\begin{remark}
It is obvious that the stability of the system depends only on the degeneracy of the right string and is independent of the degeneracy of the left string. On the other hand, we can only obtain the stability of the system in the case of  $0\leq\mu_a<1$. Mathematically, this is determined by the fact that $a^{-1}\in L^1$ is valid only when $0\leq\mu_a<1$. In practice, when the degeneracy is too strong, i.e., $1\leq\mu_a<2$, the defect at the connection point prevents the boundary control from passing through this point to control the left string.
\end{remark}

\section*{Acknowledgments}
The project is supported by the National Natural Science Foundation of China (grants No. 12271035, 12131008) and
Beijing Municipal Natural Science Foundation (grant No. 1232018).

%


Y.N. Sun, School of Mathematics and Statistics, Beijing Institute of Technology, Beijing, 100081, P.R. China

Email address: yanansun@bit.edu.cn

Q. Zhang, 
School of Mathematics and Statistics, Beijing Institute of Technology, Beijing, 100081, P.R. China

Email address: zhangqiong@bit.edu.cn

\end{document}

%% file: pagefmt.tex
\catcode`@=11
\long\def\@makefntext#1{\noindent #1}
\newskip\tabcentering\tabcentering=1000pt plus 1000pt minus 1000pt

\def\MCH#1#2{\setbox0=\hbox{\raise#1\hbox{#2}}\smash{\box0}}

\def\evenhead{}
\def\oddhead{}
\def\@evenhead{\hbox to\textwidth{\footnotesize\rm\thepage\hfill{\small\it\evenhead}}}
\def\@oddhead{\hbox to\textwidth{\footnotesize{\small\it\oddhead}\hfill\footnotesize\rm\thepage}}
\def\@evenfoot{}
\def\@oddfoot{}

\catcode`@=12
\font\bb=cmbx8%
\font\small=cmr8

\font\st=cmbx10 scaled\magstep2%
\def\title#1{{\noindent\st#1}}%
\def\author#1{\vspace*{0.3 true cm}{\noindent\bf #1}}%
\def\abstract#1{\vspace*{1 true cm}{\noindent{\footnotesize{\bb Abstract}\hspace{4true mm}#1}}}%
\def\keywords#1{\vspace{5mm}\noindent{\footnotesize{\bb Keywords}\hspace{4true mm}#1}}%
\def\mr#1{\noindent{\footnotesize{\bf 2020 MR Subject Classification}\hspace{4mm}#1}}%